\documentclass[a4paper,12pt]{amsart}
\usepackage{geometry}
\usepackage[T1, T5]{fontenc}
\usepackage{charter}
\usepackage{xcolor}
\usepackage{latexsym,amssymb,amsmath,graphicx,hyperref,setspace,tikz,subcaption,float,enumerate}
\numberwithin{equation}{section}
\newtheorem{theorem}{Theorem}[section]
\newtheorem{lemma}[theorem]{Lemma}
\newtheorem{proposition}[theorem]{Proposition}
\newtheorem{corollary}[theorem]{Corollary}

\theoremstyle{definition}
\newtheorem{definition}[theorem]{Definition} 
\newtheorem{remark}[theorem]{Remark}
\newtheorem{example}[theorem]{Example}

\newcommand{\lr}[1]{\left\langle#1\right\rangle}
\newcommand{\K}[0]{\mathbb{K}}

\DeclareMathOperator{\reg}{reg}
\DeclareMathOperator{\pdim}{pdim}

\DeclareMathOperator{\hte}{ht}

\title{The Graded Betti Numbers of the Skeletons of Simplicial Complexes}
\date{
}

\author{Mohammed Rafiq Namiq}
\address{Mohammed Rafiq Namiq, Department of Mathematics, College of Science, University of Sulaimani, Kurdistan Region, Iraq.}
\email{mohammed.namiq@univsul.edu.iq}

\makeatletter
\@namedef{subjclassname@2020}{%
	\textup{2020} Mathematics Subject Classification}
\makeatother


\begin{document}
	
\begingroup
\def\uppercasenonmath#1{} 
\let\MakeUppercase\relax 
\maketitle
\endgroup

\begin{abstract}
	In this paper, we study a class $\mathcal{C}$ of squarefree monomial ideals $I\subseteq R=\K[x_1,\dots,x_n]$ over a field $\K$, defined by the condition that $\dim R/I$ equals the maximum degree of the minimal generators of $I$ minus one. We show that the Stanley-Reisner ideal of every $i$-skeleton of a simplicial complex $\Delta$ belongs to $\mathcal{C}$ for all $-1\le i<\dim\Delta$. To investigate their homological properties, we introduce the notion of a degree resolution and prove that every ideal in $\mathcal{C}$  possesses this property. Moreover, we show that every squarefree monomial ideal admits a truncation whose regularity coincides with that of the original ideal, thereby reducing the study of degree resolutions to that of linear resolutions. Finally, we provide an explicit formula describing the relationship between the graded Betti numbers of a simplicial complex and those of its skeletons.

\end{abstract}

\section{Introduction}
Let $I$ be a squarefree monomial ideal in a polynomial ring $R=\K[x_1,\dots, x_n]$ over a field $\K$. For every squarefree monomial ideal $I\subset R$, there exists an associated simplicial complex $\Delta$ such that the quotient ring $\K[\Delta]=R/I$ is the Stanley-Reisner ring \cite{Stanley1975}. This ring encodes important combinatorial, topological and algebraic properties of the simplicial complex $\Delta$. For example, the combinatorial structure of $\Delta$ is directly encoded in the generators of the ideal $I$, which correspond to the minimal non-faces of $\Delta$. Moreover, the homological invariants of $\Delta$, such as the Betti numbers and homology groups, can be studied through the algebraic properties of the Stanley-Reisner ring $\K[\Delta]$. A minimal free resolution of $\K[\Delta]$ is of the form:
\begin{equation*}
	0\longrightarrow\bigoplus_{j}R(-j)^{\beta_{p, j}}\longrightarrow\cdots\longrightarrow\bigoplus_{j}R(-j)^{\beta_{2, j}}\longrightarrow\bigoplus_{j}R(-j)^{\beta_{1, j}}\longrightarrow R\longrightarrow \K[\Delta]\longrightarrow0.
\end{equation*}
That is, a minimal free resolution of $\K[\Delta]$ is an exact sequence of finitely generated free $R$-modules, each powers to a graded Betti number $\beta_{i,j}$. A graded Betti number $\beta_{i,j}$ counts the number of generators of degree $j$ in the $i$-th syzygy module of $\K[\Delta]$ in the minimal free resolution \cite{Peeva2011}. 

These Betti numbers are useful for understanding the algebraic complexity of $\Delta$ and play an important role in combinatorial commutative algebra and topological studies of simplicial complexes. Determining the graded Betti numbers of a minimal free resolution of a simplicial complex is often a highly challenging task, and in many cases, calculating these numbers explicitly can be extremely difficult, especially for ideals generated by many monomials.

Despite the difficulty in determining the graded Betti numbers of a minimal free resolution of a squarefree monomial ideal of $R$, another challenge arises in understanding how these graded Betti numbers change under certain operations on 
$I$. In particular, for a given integer $k\ge0$, this includes studying the graded Betti numbers of the following related ideals:
\begin{enumerate}
	\item The $k$-th power of $I$, denoted by $I^k$.
	\item The ideal generated by the (squarefree) monomials in $I$ whose degree is at least $k$, denoted by ($I_k$) $I_{\ge k}$.
	\item The Stanley-Reisner ideal corresponding to the $k$-skeleton of the simplicial complex $\Delta$, denoted by $I_{\Delta^k}$.
\end{enumerate}
The first case has been studied by many researchers for certain classes of ideals, such as in \cite{Banerjee2015,BeyarslanHaTrung2015,CutkoskyHerzogTrung1999, HerzogHibiZheng20042,JayanthanSelvaraja2021,NevoPeeva2013,Romer2001}. However, there is currently no general formula for computing the graded Betti numbers of $I^k$.  These numbers appear to lack pattern. For example, the ideal $I=(acf,ade,bcd,bef,cdf,cde,cef,def)$, as presented by Sturmfels in \cite{Sturmfels2000}, has a linear resolution, but $I^2$ does not have a linear resolution.

The second case was studied by Ahmed, Fröberg, and Namiq \cite{AhmedFrobergNamiq2023}, who provided an explicit formula for computing the graded Betti numbers of $I_k$ and $I_{\ge k}$ in terms of those of $I$. This result is particularly significant because $\reg I=k$ if and only if $k$ is the smallest integer such that $I_{\ge k}$ $(I_k)$ has a linear resolution. Furthermore, once $I_{\ge k}$ $(I_k)$ admits a linear resolution, the ideals $I_{\ge\ell}$ $(I_\ell)$ also have linear resolutions for all $\ell\geq k$. We note that $I^k\subseteq I_{\ge s}$ for all $s\le k\alpha(I)$, where $\alpha(I)$ is the minimum degree of the minimal generators of $I$.

In this paper, our goal is to study the third case: the relationship between the graded Betti numbers of a full simplicial complex and those of its skeletons. Specifically, we aim to determine whether the graded Betti numbers of a skeleton can provide meaningful insights into the graded Betti numbers of the original complex and vice versa. To this end, we provide both necessary and sufficient conditions, along with formulas for computing the graded Betti numbers in both directions.





This paper is organised as follows: We begin by reviewing the foundational concepts of simplicial complexes. We then introduce the notion of a degree resolution (see Definition \ref{degree resolution}), which occurs when the regularity of $I$ equals the degree of $I$. Here, the degree of $I$, denoted by $\omega(I)$, is the maximum degree among the minimal generators of $I$. Then we establish a relation between the dimension $\dim R/I$ and the degree $\omega(I)$. Specifically, in Lemma \ref{dimension degree}, we show that 
$$\dim R/I\ge\omega(I)-1.$$
Equality holds if and only if the ideal $I$ has a degree resolution and $\reg R/I=\dim R/I$. This result establishes a relation between the combinatorial properties of the simplicial complex and the algebraic properties of the ideal. As a consequence of this result, we show that $I_{\Delta^i}$ has a degree resolution and $\reg\K[\Delta^i]=\dim\K[\Delta^i]$ for any simplicial complex $\Delta$ and for $-1\leq i<\dim\Delta$. The $i$-skeleton of a simplicial complex is the subcomplex consisting of all simplices of dimension at most $i$.

Next, we study the projective dimension of monomial ideals and connect it to the $i$-skeleton of the associated simplicial complex, as shown in Lemma \ref{pdim}. Finally, in Theorem \ref{Betti k-skeleton} and Corollaries \ref{Betti deg1} and \ref{Betti deg}, we present formula for calculating the graded Betti numbers of the $i$-skeleton from the original complex, and vice versa. These results show a strong relationship between the graded Betti numbers of a simplicial complex and its skeletons. In the Betti table of $I_{\Delta^i}$ for $-1\le i<\dim\Delta$, the only row that differs from the Betti table of $I_{\Delta}$ is the one corresponding to $\omega(I_{\Delta^i})$, with all rows below $\omega(I_{\Delta^i})$ being zero.

\subsection*{Acknowledgments}
The author expresses sincere thanks to Dr. Chwas Ahmed for his careful reading of this work and for his many valuable comments.

\section{Preliminaries}
Fix $n>0$ and let $X=\{x_1,x_2,\dots,x_n\}$ be a finite vertex set and $R=\K[x_i\mid x_i\in X]$ be a polynomial ring in $n$ variables over a field $\K$. A \emph{simplicial complex} $\Delta$ on the vertex set $X$, is a collection of subsets of $X$ that satisfies the following two conditions:
\begin{enumerate}
	\item For every $x_i\in X$, the singleton set $\{x_i\}$ is in $\Delta$.
	\item If $F$ is in $\Delta$, then any subset $F'$ of $F$ is also in $\Delta$.
\end{enumerate}

An element $F$ in $\Delta$ is called a \emph{face} of $\Delta$, and its dimension is defined as $|F|-1$, which we denote as $\dim F$. A face of dimension $i$ is referred to as an \emph{$i$-face}. The collection of all $0$-faces in $\Delta$ is the vertex set. 

A \emph{facet} is a maximal face (with respect to inclusion) of $\Delta$. We denote the set of all facets of $\Delta$ by $\mathcal{F}(\Delta)$, and we sometimes write $\Delta=\langle F\mid F\in\mathcal{F}(\Delta)\rangle$. Moreover, a subset $N$ of $X$ is a \emph{non-face} of $\Delta$ if it is not a face of $\Delta$. The set of all minimal non-faces (with respect to inclusion) is denoted as $\mathcal{N}(\Delta)$. 

The \emph{dimension} of $\Delta$ is defined as $\dim\Delta=d-1$, where $d=\max\{|F|\mid F\in\mathcal{F}(\Delta)\}$. A simplicial complex $\Delta$ is called \emph{pure} if all its facets have the same cardinality. 

A \emph{subcomplex} $\Gamma$ of $\Delta$ is a simplicial complex whose facets are faces of $\Delta$. The \emph{$i$-skeleton} $\Delta^i$ of $\Delta$ is the subcomplex containing faces of dimension $i$ or less.

The \emph{Stanley-Reisner ideal} $I_\Delta$ of a simplicial complex $\Delta$ is defined as:
$$I_\Delta=\left(\prod_{x_i\in N}x_i\mid N\in\mathcal{N}(\Delta)\right).$$
The \emph{Stanley-Reisner ring} $\K[\Delta]$ is the quotient ring $R/I_{\Delta}$.

The \emph{$f$-vector} of a simplicial complex $\Delta$ is a sequence $f(\Delta)=(f_{-1},f_0,\ldots,f_{d-1})$, where $f_i$ represents the number of $i$-faces of $\Delta$. The \emph{$h$-vector} of $\Delta$ is the $d$-tuple $h(\Delta)=(h_0,\ldots,h_d)$, where the integers $h_i$ can be determined by the following relation $$\sum_{i=0}^{d}f_{i-1}t^i(1-t)^{d-i}=\sum_{i=0}^{d}h_it^i.$$

A \emph{monomial} $u$ in $R$ is a product of the form $x_1^{a_1}x_2^{a_2}\dots x_n^{a_n}$, where $\mathbf{a}=(a_1,a_2,\dots,a_n)\in\mathbb{N}^n$ is a vector of non-negative integers. The \emph{degree} of the monomial $u$ is defined as $\deg(u)=a_1+a_2+\dots+a_n$. A \emph{monomial ideal} of $R$ is an ideal generated by monomials in $R$. 

Given any set of monomial generators for a monomial ideal $I$, it is possible to eliminate any monomials that are divisible by others in the set without changing the generating set of $I$. This process produces the unique minimal set of monomials that generates $I$. These monomials are called the \emph{minimal generators} of $I$, and the set of minimal generators of $I$ is denoted by $\mathcal{G}(I)$.


The \emph{initial degree} of $I$, denoted by $\alpha(I)$, is the minimum degree of the minimal generators of $I$, and the \emph{degree} of $I$, denoted by $\omega(I)$, is the maximum degree of the minimal generators of $I$. That is,  
$$\alpha(I)=\min\{\deg(u)\mid u\in\mathcal{G}(I)\}\quad\text{and}\quad\omega(I)=\max\{\deg(u)\mid u\in\mathcal{G}(I)\}.$$

Let $I\subseteq R$ be a monomial ideal, consider the following minimal graded free resolution of $R/I$:
$$0\longrightarrow\bigoplus_{j}R(-j)^{\beta_{p, j}}\longrightarrow\cdots\longrightarrow\bigoplus_{j}R(-j)^{\beta_{2, j}}\longrightarrow\bigoplus_{j}R(-j)^{\beta_{1, j}}\longrightarrow R\longrightarrow R/I\longrightarrow0,$$
where $R(-j)$ denotes the $R$-module obtained by shifting the degrees of each generator of $R$ by $j$. The integer $\beta_{i,j}(R/I):=\beta_{i,j}$ is called the \emph{$i$-th graded Betti number} of $R/I$ in degree $j$. The length $p$ of the resolution is called the \emph{projective dimension} of $R/I$ over $R$, and can be defined as:
$$\pdim R/I=\max\{i\mid\beta_{i,j}\neq 0\text{ for some }j\}.$$
The \emph{regularity} of $R/I$ over $R$, also known as \emph{Castelnuovo–Mumford regularity}, is:
$$\reg R/I=\max\{j-i\mid\beta_{i,j}\neq 0\}.$$



An important lemma that we will use later is Hochster's formula. This formula offers a crucial understanding of the graded Betti numbers of a Stanley-Reisner ring (see {\cite[Theorem 5.1]{Hochster1977}} or {\cite[Lemma 9]{Froberg2021}}). Denote by $\Delta_W$ the simplicial complex on the vertex set $W\subseteq\{x_1,\ldots,x_n\}$, where the faces of $\Delta_W$ are those $F\in\Delta$ such that $F\subseteq W$.

\begin{lemma}[\textbf{Hochster's formula}]\label{hochster}
	Let $W\subseteq\{x_1,\ldots,x_n\}$ and let $K_{R(W)}$ be the part of the Koszul complex $K_R$ that corresponds to the degree $\delta(W)=(d_1,\ldots,d_n)$, where $d_i=1$ if $x_i\in W$ and $d_i=0$ otherwise. Then 
	$$H_{i,\delta(W)}=H_i(K_{R(W)})\cong\tilde{H}_{|W|-i-1}(\Delta_W;\K).$$
\end{lemma}

For any unfamiliar or unexplained terminologies, we refer the reader to \cite{BrunsHerzog1998,Eisenbud1995, Villarreal2015} and the reference therein.

\section{The graded Betti numbers of the $i$-skeletons of a simplicial complex}

A monomial ideal with a linear resolution requires all its minimal generators to have the same degree. To generalize this concept, we introduce the notion of a degree resolution, which applies to monomial ideals whose minimal generators may have different degrees. We then derive an equation that allows us to calculate the graded Betti numbers of the $i$-skeletons of a given simplicial complex $\Delta$, utilizing the graded Betti numbers of the original complex $\Delta$. We begin by providing the following definition.

\begin{definition}\label{degree resolution}
	A monomial ideal $I$ in the polynomial ring $R$ is said to have a \emph{degree resolution} if $\beta_{i,j}(I)=0$ for all $i\ge0$ and $j>i+\omega(I)$. Equivalently, the regularity of $I$ satisfies  
	$$\reg I=\max\{j\mid \beta_{0,j}\ne0\}.$$
	In other words, the regularity of $I$ is equal to its degree, i.e., $\reg I = \omega(I)$.
\end{definition}

For example, the ideal $I=(x_1x_2,x_2x_3x_4,x_3x_4x_5)$ has a degree resolution because both its regularity and degree are equal, that is, $\reg I=\omega(I)=3$. On the other hand, the ideal $J=(x_1x_2,x_3x_4)$ does not have a degree resolution since $\omega(J)=2$ while $\reg J=3$.

\begin{remark}
	The notion of degree resolution can be extended to any finitely generated graded $R$-module.
\end{remark}

The following result shows that every squarefree monomial ideal admits a truncation whose regularity equals that of the ideal itself.


\begin{proposition}\label{degree resolution and I_k}
	A squarefree monomial ideal $I$ has a degree resolution if and only if the truncation $I_k$ has a linear resolution, where $k=\omega(I)$.
\end{proposition}
\begin{proof}
	Suppose first that $I$ has a degree resolution. By definition, this means that $\reg I=\omega(I)$. Let $k=\omega(I)$. We observe that the truncation $I_k$ contains only the generators of degree $k$. By \cite[Theorem 2.5]{AhmedFrobergNamiq2023}, the only nonzero row in the Betti table of $I_k$ corresponds to degree $k$. This shows that $I_k$ has a linear resolution.
	
	Conversely, assume that $I_{k}$ has a linear resolution for $k=\omega(I)$. Then $\reg I_k=k=\omega(I)$. Since $\reg I\ge\omega(I)$ always holds, and by \cite[Theorem 2.5]{AhmedFrobergNamiq2023} the linearity of $I_k$ implies $\reg I=k$. It follows that $\reg I=\omega(I)$. Therefore, $I$ has a degree resolution.
\end{proof}

\begin{remark}
	Proposition \ref{degree resolution and I_k} shows that the property of having a degree resolution can be fully studied via the truncation $I_k$, where $k=\omega(I)$. In particular, this implies that for any squarefree monomial ideal $I$, the study of its regularity can be reduced to that of the truncation $I_k$. Consequently, when working with squarefree monomial ideals generated in different degrees, it suffices to consider the corresponding truncation ideals generated in a single degree, since the regularity of the truncation coincides with that of the original ideal.
\end{remark}

\begin{corollary}
	A monomial ideal $I\subseteq R$ has a linear resolution if and only if $I$ has a degree resolution and all of its minimal generators have the same degree.
\end{corollary}
\qed

Next, we present a lemma regarding the relationship between the dimension and degree of a squarefree monomial ideal. This result establishes a connection between the dimension of $\K[\Delta]$, the degree of $I_{\Delta}$, and the condition under which $I_{\Delta}$ has a degree resolution.
\begin{lemma}\label{dimension degree}
	Let $\Delta$ be a simplicial complex. Then we have
	$$\dim\K[\Delta]\ge\omega(I_{\Delta})-1.$$
	The equality holds if and only if $I_{\Delta}$ has a degree resolution and $\reg\K[\Delta]=\dim\K[\Delta]$.
\end{lemma}
\begin{proof}
	Each monomial generator of the ideal $I_{\Delta}$ corresponds to a minimal non-face of the simplicial complex $\Delta$. The degree of $I_{\Delta}$ is determined by the largest minimal non-face of $\Delta$.
	
	A minimal non-face of $\Delta$ is a set of vertices that is not a face of $\Delta$, but all of its proper subsets are faces of $\Delta$. Consequently, the dimension of $\K[\Delta]$ satisfies $\dim\K[\Delta]=\dim\Delta+1\geq\omega(I_{\Delta})-1$.
	
	Now, assume that $\dim\K[\Delta]=\omega(I_{\Delta})-1$. Under this assumption and \cite[Theorem 3.1]{HoaTrung1998}, we have $\reg\K[\Delta]\le n-\hte I_{\Delta}=\omega(I_{\Delta})-1$. On the other hand, $\reg\K[\Delta]\ge\omega(I_\Delta)-1$. Thus, it follows that $\reg\K[\Delta]=\omega(I_{\Delta})-1$, which means that $I_{\Delta}$ has a degree resolution. Consequently, this implies $\reg\K[\Delta]=\dim\K[\Delta]$.
	
	Conversely, let $\omega(I_{\Delta})-1=\reg\K[\Delta]=\dim\K[\Delta]$. Then the result immediately follows.
\end{proof}

As a consequence of Lemma \ref{dimension degree}, we obtain the following corollary regarding the degree resolution of the $i$-skeleton. This corollary is sufficient to establish that the $i$-skeleton of any simplicial complex with $i<d-1$ has a degree resolution, and that the regularity is equal to the dimension.

\begin{corollary}\label{skeleton degree resolution}
	Let $\Delta$ be a simplicial complex. Then $I_{\Delta^i}$ has a degree resolution and $\reg\K[\Delta^i]=\dim\K[\Delta^i]$ for $-1\le i<\dim\Delta$.
\end{corollary}
\begin{proof}
	We have $\dim\K[\Delta^i]=i+1$ and $\omega(I_{\Delta^i})=i+2$ for $-1\le i<\dim\Delta$. It follows from Lemma \ref{dimension degree} that $I_{\Delta^i}$ has a degree resolution.
\end{proof}



Now, we present the following lemma that links the projective dimension of a simplicial complex to its skeletons. This lemma shows how the projective dimension varies when moving to a smaller skeleton.
\begin{lemma}\label{pdim}
	Let $\Delta$ be a simplicial complex. If $i\ge0$ is an integer such that $\Delta^{i}$ is pure and $\pdim\K[\Delta^i]=p$, then $\pdim\K[\Delta^{i-1}]=p+1$.
\end{lemma}
\begin{proof}
	From \cite{AhmedFrobergNamiq2023}, we note that $\Delta^{i-1}$ is Cohen-Macaulay if $\Delta^i$ is Cohen-Macaulay. Furthermore, we have $\hte I_{\Delta^{i-1}}=\hte I_{\Delta^i}+1$. Therefore, the conclusion follows from these property.
\end{proof}

The following theorem provides a formula for calculating the graded Betti numbers of the skeletons of a simplicial complex. This formula based on the $f$-vector and the Betti numbers of the original complex. This result implies that in the Betti table of $\K[\Delta^i]$ for $-1 \leq i<\dim\Delta$, the only row that differs from the Betti table of $\K[\Delta]$ is the one corresponding to $\omega(I_{\Delta^i})-1$, with all rows below $\omega(I_{\Delta^i})-1$ being zero.

\begin{theorem}\label{Betti k-skeleton}
	Let $\Delta$ be a simplicial complex. For $-1\le k<\dim\Delta$, the graded Betti numbers of $\K[\Delta^k]$ can be determined by the following equation:
	$$(-1)^{i}\beta_{i,s}(\K[\Delta^k])=\sum_{r=0}^{s}(-1)^{s-r}\binom{n-r}{s-r}f_{r-1}(\Delta^k)-\sum_{i'>i}(-1)^{i'}\sum_{i'+j'=s}\beta_{i',s}(\K[\Delta]),$$
	where $s=i+j$, $j=\omega(I_{\Delta^k})-1$ and $s=\omega(I_{\Delta^k})-1,\omega(I_{\Delta^k}),\dots,n$.
\end{theorem}
\begin{proof}
	Let $f(\Delta)=(f_{-1},f_0,\dots,f_{d-1})$ be the $f$-vector of $\Delta$. The $f$-vector of the $k$-skeleton of $\Delta$, is then $(f_{-1},f_0,\dots,f_{k-1})$, where $-1\le k\le\dim\Delta$. Hence by Hochster’s formula, Lemma \ref{hochster}, the graded Betti numbers $\beta_{i,i+j}(\K[\Delta])$ remain unchanged for $j<\omega(I_{\Delta^k})-1$, i.e. $\beta_{i,i+j}(\K[\Delta^k])=\beta_{i,i+j}(\K[\Delta])$ for $j<\omega(I_{\Delta^k})-1$. The Hilbert series of the quotient ring $\K[\Delta]$ is given by
	$$H_{\K[\Delta]}(t)=\sum_{r=0}^{d} \frac{f_{r-1} t^{r}}{(1-t)^{r}}=\frac{\sum_i(-1)^i\sum_j\beta_{i,i+j}t^{i+j}}{(1-t)^n}.$$
	Moreover, we have that
	$$\displaystyle\sum_{r=0}^{d} \frac{f_{r-1} t^{r}}{(1-t)^{r}}\times(1-t)^n=\sum_{s=0}^{n}\sum_{r=0}^{s}(-1)^{s-r}\binom{n-r}{s-r}f_{r-1}t^s.$$
	As a result, from the Hilbert series of $\K[\Delta]$ we obtain
	$$\sum_{s=0}^{n}\sum_{r=0}^{s}(-1)^{s-r}\binom{n-r}{s-r}f_{r-1}t^s=\sum_{i}(-1)^i\sum_j\beta_{i,i+j}t^{i+j}$$
	By substituting $s=i+j$, we can derive
	$$\sum_{r=0}^{s}(-1)^{s-r}\binom{n-r}{s-r}f_{r-1}=\sum_{i}(-1)^i\sum_{i+j=s}\beta_{i,s},\quad s=0,1,\dots,n.$$
	From Corollary \ref{skeleton degree resolution}, it follows that $I_{\Delta^k}$ has a degree resolution for $-1\le k<\dim\Delta$. Therefore, we conclude that $j=\omega(I_{\Delta^k})-1$ and
	$$(-1)^{i}\beta_{i,s}(\K[\Delta^k])=\sum_{r=0}^{s}(-1)^{s-r}\binom{n-r}{s-r}f_{r-1}(\Delta^k)-\sum_{i'>i}(-1)^{i'}\sum_{i'+j'=s}\beta_{i',s}(\K[\Delta]),$$
	where $s=\omega(I_{\Delta^k})-1,\omega(I_{\Delta^k}),\dots,n$.
\end{proof}

\begin{remark}
	\begin{enumerate}
		\item For any Stanley-Reisner ring $\K[\Delta]$ with a linear resolution, the graded Betti numbers of $\K[\Delta]$ are determined by the formula:
	$$\beta_{i,i+j}(\K[\Delta])=\sum_{r=0}^{i+j}(-1)^{j-r}\binom{n-r}{i+j-r}f_{r-1}(\Delta), j=\omega(I_{\Delta})-1, i+j=\omega(I_{\Delta}),\omega(I_{\Delta})+1,\dots,n.$$
	This formula depends on the combinatorial properties of the simplicial complex $\Delta$, particularly the $f$-vector of $\Delta$ and the degree of $I_{\Delta}$. Consequently, the graded Betti numbers of $\Delta$ and all its skeletons can be directly computed using this formula along with the formula in Theorem \ref{Betti k-skeleton}.
	
	\item The key difference between this work and that of Ahmed et al. \cite{AhmedFrobergNamiq2023} lies in the procedure used to compute the graded Betti numbers. In their approach, to calculate the graded Betti numbers of $I_k$, all squarefree monomials of degree $k$ from $I$ are added to the set of minimal generators, and any monomial of degree $k-1$ in the minimal generators of $I$ is removed. In contrast, this work computes the graded Betti numbers of $I_{\Delta^k}$ by adding all faces of the simplicial complex associated with $I$ of cardinality $k+2$ to $I$. For example, consider the ideal $I=(ab,bcd,cdef)$. The associated simplicial complex of $I$ is $\Delta=\lr{acde,acdf,acef,bcef,adef,bdef}$. The ideal $I_3$ is $(abc,abd,abe,adf,bcd,cdef)$. Moreover, the ideal $I_{\Delta^{2}}=(ab,bcd,acde, acdf,acef,bcef,adef,bdef,cdef)$.
	\end{enumerate}
\end{remark}

To illustrate the results discussed above, we present the following detailed example.
\begin{example}
	Consider the ideal $I=(x_1x_2,x_2x_3x_4,x_5x_6x_7x_8,x_1x_3x_5x_7x_9x_{10})$ of the polynomial ring $R=\K[x_1,\dots,x_{10}]$. The $f$-vector of the simplicial complex $\Delta$ associated with $I$ is $f(\Delta)=(1,10,44,111,175,175,105,31,2)$. Using Macaulay2 \cite{M2}, we can compute the Betti table of $\K[\Delta]$, which is given by:
	\begin{center}
		\begin{tabular}{c|cccc} 
			& 0 & 1 & 2 & 3 \\
			\hline 0 & 1 & 0 & 0 & 0 \\
			1 & 0 & 1 & 0 & 0 \\
			2 & 0 & 1 & 1 & 0 \\
			3 & 0 & 1& 0 & 0 \\
			4 & 0 & 0 & 1 & 0 \\
			5 & 0 & 1 & 2 & 1\\
			6 & 0 & 0 & 1 & 1
		\end{tabular}
	\end{center}
	Next, we compute the graded Betti numbers of $\K[\Delta^4]$, where $\Delta^4$ is the 4-skeleton of $\Delta$. The $f$-vector of $\Delta^4$ is $f(\Delta^4)=(1,10,44,111,175,175)$.
	According to Corollary \ref{skeleton degree resolution}, we have $j=\omega(I_{\Delta^4})-1=\dim\K[\Delta^4]=5$. Since $\pdim\K[\Delta]=3=\hte I_{\Delta^6}$, by Lemma \ref{pdim}, we have $\pdim\K[\Delta^4]=5$. Now, using Theorem \ref{Betti k-skeleton}, we can compute the graded Betti numbers of $\K[\Delta^4]$ with the formula:
	$$(-1)^{i}\beta_{i,s}(\K[\Delta^k])=\sum_{r=0}^{s}(-1)^{s-r}\binom{n-r}{s-r}f_{r-1}(\Delta^k)-\sum_{i'>i}(-1)^{i'}\sum_{i'+j'=s}\beta_{i',s}(\K[\Delta]),$$
	where $i=0,\dots,5$, and $s=i+j$. We now compute each graded Betti number of $\K[\Delta^4]$ step by step:
	\begin{itemize}
		\item For $\beta_{1,6}(\K[\Delta^4])$:
		$$(-1)^1 \beta_{1,6}(\K[\Delta^4])=\sum_{r=0}^{6}(-1)^{6-r} \binom{10-r}{6-r}f_{r-1}(\Delta^4)-\beta_{2,6}(\K[\Delta]),$$
		which simplifies to:
		$$(-1)^6\binom{10}{6}\times1+(-1)^5\binom{9}{5}\times 10+\dots+(-1)^1\binom{5}{1}\times175-1=-106.$$
		Thus, $\beta_{1,6}(\K[\Delta^4])=106$.
		
		\item For $\beta_{2,7}(\K[\Delta^4])$:
		$$(-1)^2\beta_{2,7}(\K[\Delta^4])=\sum_{r=0}^{7}(-1)^{7-r}\binom{10-r}{7-r}f_{r-1}(\Delta^4),$$
		which simplifies to:
		$$(-1)^7\binom{10}{7}\times1+(-1)^6\binom{9}{6}\times10+\dots+(-1)^2\binom{5}{2}\times175=391.$$
		Thus, $\beta_{2,7}(\K[\Delta^4])=391$.
		
		\item For $\beta_{3,8}(\K[\Delta^4])$:
		$$(-1)^3\beta_{3,8}(\K[\Delta^4])=\sum_{r=0}^{8}(-1)^{8-r}\binom{10-r}{8-r}f_{r-1}(\Delta^4),$$
		which simplifies to:
		$$(-1)^8\binom{10}{8}\times1+(-1)^7\binom{9}{7}\times10+\dots+(-1)^3\binom{5}{3}\times175= -539.$$
		Thus, $\beta_{3,8}(\K[\Delta^4])=539$.
		
		\item For $\beta_{4,9}(\K[\Delta^4])$:
		$$(-1)^4\beta_{4,9}(\K[\Delta^4])=\sum_{r=0}^{9}(-1)^{9-r}\binom{10-r}{9-r} f_{r-1}(\Delta^4),$$
		which simplifies to:
		$$(-1)^9\binom{10}{9}\times1+(-1)^8\binom{9}{8}\times10+\dots+(-1)^4\binom{5}{4}\times175=330.$$
		Thus, $\beta_{4,9}(\K[\Delta^4])=330$.
		
		\item For $\beta_{5,10}(\K[\Delta^4])$:
		$$(-1)^5\beta_{5,10}(\K[\Delta^4])=\sum_{r=0}^{10}(-1)^{10-r}\binom{10-r}{10-r}f_{r-1}(\Delta^4),$$
		which simplifies to:
		$$(-1)^{10}\binom{10}{10}\times1+(-1)^9\binom{9}{9}\times10+\dots+(-1)^5\binom{5}{5}\times175=-76.$$
		Thus, $\beta_{5,10}(\K[\Delta^4])=76$.
	\end{itemize}
	Therefore, the Betti table of $\K[\Delta^4]$ is:
	$$\begin{array}{c|cccccc}
		& 0 & 1 & 2 & 3 & 4 & 5 \\
		\hline
		0 & 1 & 0 & 0 & 0 & 0 & 0 \\
		1 & 0 & 1 & 0 & 0 & 0 & 0 \\
		2 & 0 & 1 & 1 & 0 & 0 & 0 \\
		3 & 0 & 1 & 0 & 0 & 0 & 0 \\
		4 & 0 & 0 & 1 & 0 & 0 & 0 \\
		5 & 0 & 106 & 391 & 539 & 330 & 76
	\end{array}$$
\end{example}

In the next two corollaries, we present a way to compute the graded Betti numbers of a simplicial complex by using the graded Betti numbers of one of its skeleton. These corollaries are direct consequences of Theorem \ref{Betti k-skeleton}.

\begin{corollary}\label{Betti deg1}
	Let $\Delta$ be a simplicial complex. Then $\reg I_{\Delta}<\omega(I_{\Delta^i})$ for some $-1\le i<\dim\Delta$ if and only if the graded Betti numbers of $\K[\Delta]$ can be directly determined from the graded Betti numbers of $\K[\Delta^i]$.
\end{corollary}
\begin{proof}
	Let $\reg I_{\Delta}<\omega(I_{\Delta^i})$ for some $-1\leq i< \dim\Delta$. By Theorem \ref{Betti k-skeleton}, this inequality holds if and only if none of the graded Betti numbers of $\K[\Delta^i]$ are changed. Consequently, $\reg I_{\Delta}<\omega(I_{\Delta^i})$ for some $-1\leq i<\dim\Delta$ if and only if the graded Betti numbers of $\K[\Delta]$ can be determined directly from the graded Betti numbers of $\K[\Delta^i]$.
\end{proof}

For instance, in the previous example, we cannot determine the graded Betti numbers of $\K[\Delta]$ form the Betti table of $\K[\Delta^4]$ because $\reg I_{\Delta}=7>\omega(I_{\Delta^4})=6$.

\begin{corollary}\label{Betti deg}
	Let $\Delta$ be a simplicial complex. Then $\reg I_{\Delta}=\omega(I_{\Delta^i})$ for some $-1\le i<\dim\Delta$ if and only if the graded Betti numbers of $\K[\Delta]$ can be determined from the graded Betti numbers of $\K[\Delta^i]$ and the $h$-vector of $\Delta$.
\end{corollary}
\begin{proof}
	Let $\reg I_{\Delta}=\omega(I_{\Delta^i})$ for a fixed $-1\le i<\dim\Delta$. According to Theorem \ref{Betti k-skeleton}, this equality holds if and only if none of the graded Betti numbers of $\K[\Delta^i]$ is changed for all $j<\reg\K[\Delta]$, and the only row that changes corresponds to $j=\reg\K[\Delta]$. Consequently, $\reg I_{\Delta}=\omega(I_{\Delta^i})$ if and only if the graded Betti numbers of $\K[\Delta]$ can be entirely determined from the graded Betti numbers of $\K[\Delta^i]$ and the $h$-vector of $\Delta$.
\end{proof}

\begin{example}
	As in the previous example, let  $I=(x_1x_2,x_2x_3x_4,x_5x_6x_7x_8,x_1x_3x_5x_7x_9x_{10})$ be an ideal of the polynomial ring $R=\K[x_1,\dots,x_{10}]$. The $h$-vector of the simplicial complex $\Delta$ associated with $I$ is given by $h(\Delta)=(1,2,2,1,0,-1,-2,-1,0)$. Using Macaulay2, we can compute the Betti table of $\K[\Delta^5]$, which is as follows:
	\begin{center}
		\begin{tabular}{c|ccccc} 
			& 0 & 1 & 2 & 3 & 4\\
			\hline 0 & 1 & 0 & 0 & 0 & 0\\
			1 & 0 & 1 & 0 & 0 & 0\\
			2 & 0 & 1 & 1 & 0 & 0\\
			3 & 0 & 1& 0 & 0 & 0\\
			4 & 0 & 0 & 1 & 0 & 0\\
			5 & 0 & 1 & 2 & 1 & 0\\
			6 & 0 & 31 & 92 & 90 & 29
		\end{tabular}
	\end{center}
	By Corollary \ref{Betti deg}, the graded Betti numbers of $\K[\Delta]$ can be derived from the graded Betti numbers of $\K[\Delta^5]$, since $\reg I_{\Delta}=7=\omega(I_{\Delta^5})$. Thus, we have $$\left(\sum_{i=0}^{8}h_it^i\right)(1-t)^2=\sum_{i=0}^{4}(-1)^i\sum_{j=0}^{6}\beta_{i,i+j}(\K[\Delta])t^{i+j}$$
	Therefore, $(1-t^2-t^3+2t^7-t^9)=\sum_{i=0}^{4}(-1)^i\sum_{j=0}^{6}\beta_{i,i+j}(\K[\Delta])t^{i+j}$. From this, we find that $-\beta_{1,7}+\beta_{2,7}=2$, implying that  $\beta_{1,7}=0$. Additionally, since $\beta_{2,8}-\beta_{3,8}=0$, we get $\beta_{2,8}=1$. Moreover, $-\beta_{3,9}+\beta_{4,9}=-1$, so $\beta_{3,9}=1$. Finally, $\beta_{4,10}=0$. Thus, the Betti table of $\K[\Delta]$ is as follows:
	\begin{center}
		\begin{tabular}{c|cccc} 
			& 0 & 1 & 2 & 3 \\
			\hline 0 & 1 & 0 & 0 & 0 \\
			1 & 0 & 1 & 0 & 0 \\
			2 & 0 & 1 & 1 & 0 \\
			3 & 0 & 1& 0 & 0 \\
			4 & 0 & 0 & 1 & 0 \\
			5 & 0 & 1 & 2 & 1\\
			6 & 0 & 0 & 1 & 1
		\end{tabular}
	\end{center}
\end{example}

In conclusion, for all integers $r\ge2$, $s>\alpha(I)$, and $-1\leq t<\dim\Delta$, the Betti table of $I^r$ cannot generally be determined from the Betti table of $I$. However, the Betti tables of $I_s$ and $I_{\Delta^t}$ can be obtained from the Betti table of $I$. In particular, if $I$ has a degree resolution, $I^r$ does not necessarily have a degree resolution, but $I_s$ does have a degree resolution. Furthermore, $I_{\Delta^t}$ always has a degree resolution, regardless of whether $I$ has a degree resolution.



\textbf{Data availability.} Author can confirm that all relevant data are included in the article.

\section{Statements and Declarations}
There is no conflicts of interest to disclose.

\bibliographystyle{plain}
\bibliography{References}

\end{document}